\documentclass[12pt]{article}
\usepackage[utf8]{inputenc}
\usepackage{amsmath, amssymb, amsthm, amscd}
\usepackage[english]{babel}
\usepackage{geometry}
\usepackage{color}
\usepackage{mathrsfs}

\newtheorem{theorem}{Theorem}[section]
\newtheorem{lemma}[theorem]{Lemma}
\newtheorem{proposition}[theorem]{Proposition}
\newtheorem{corollary}[theorem]{Corollary}
\theoremstyle{definition}
\newtheorem{definition}[theorem]{Definition}

\newtheorem{example}[theorem]{Example}

\title{Affine Anosov Maps on $\mathbb{R}^n$: Classification, Index Spectrum, and Stability at Infinity}

\author{Z. Li, A. Rojas, and S. Roma\~na}

\date{}

\begin{document}

\maketitle

\begin{abstract}
For $n\ge2$, we classify the affine diffeomorphisms
$f_{A,v}(x)=Ax+v$ on $\mathbb R^n$ that admit a complete Riemannian metric with respect to which they are Anosov. Such a metric exists if and only if $A$ is hyperbolic or $f_{A,v}$ has no fixed point, equivalently $v\notin\operatorname{Im}(I-A)$. In the latter case, $f_{A,v}$ is smoothly conjugate to a translation when $\det A>0$ and to White's map times the identity when $\det A<0$. We also determine the possible stable indices. In the hyperbolic case the index is determined by the stable spectrum of $A$, whereas a map with no fixed point admits complete Anosov metrics of every stable index from $1$ to $n-1$. Along the $1$-eigenspace, every such metric must exhibit exponential growth of vector norms along one of the two half orbits. We then determine the interior and boundary of the Anosov-realizable locus in the affine parameter space and describe the corresponding change of the index spectrum near regular drift parameters. Finally, we show that Anosov-realizability is not open in the two-sided weak $C^1_{\mathrm{loc}}$ topology but is open in the two-sided strong Whitney $C^1$ topology.
\end{abstract}

\section{Introduction}

Anosov diffeomorphisms \cite{anosov} form a basic class of uniformly hyperbolic systems. On compact manifolds, Anosov diffeomorphisms are structurally stable and exhibit uniform hyperbolic behavior. Their general classification remains a major problem, although strong classification results are known in important algebraic settings (see \cite{hirsch, Katok, franks, smale} for more details). On noncompact manifolds, compactness no longer provides uniform geometric control, and several standard arguments must be reformulated.

White \cite{white} (see also \cite{Mats, Mendes} and the introduction of \cite{BenOvadia} for more details) gave a fundamental example on $\mathbb{R}^2$:
\[
T(x,y)=(x+1,-y).
\]
This map is not hyperbolic in the Euclidean metric, but becomes Anosov after a suitable change of Riemannian metric. White's example shows that affine maps with eigenvalues on the unit circle can still be Anosov, provided there is a non-zero translation in the $1$-direction. Throughout the discussion of affine maps on $\mathbb R^n$, we assume
$n\geq2$.

We classify the affine diffeomorphisms
\[
f_{A,v}(x)=Ax+v,\qquad A\in GL(n,\mathbb{R}),\quad v\in\mathbb{R}^n,
\]
on $\mathbb{R}^n$ that admit a complete Riemannian metric making them Anosov.

The criterion involves both the spectrum of $A$ and the solvability of the equation $Ax+v=x$. The main result is:

\begin{theorem}[Main Classification]\label{thm:mainresult}
Let $f_{A,v}(x)=Ax+v$ be an affine diffeomorphism of $\mathbb{R}^n$ for $n\geq 2$, with $A\in GL(n,\mathbb{R})$. Then $f_{A,v}$ admits a complete Riemannian metric making it Anosov if and only if either
\begin{enumerate}
\item[\emph{\text{1.}}] $A$ is hyperbolic, i.e., $\operatorname{Spec}(A)\cap \mathbb{S}^1=\varnothing$; or
\item[\emph{\text{2.}}] $f_{A,v}$ has no fixed point, equivalently $v\notin \operatorname{Im}(I-A)$.
\end{enumerate}
\end{theorem}
If $A$ is hyperbolic, then every affine perturbation $x\mapsto Ax+v$ is Anosov with respect to an adapted Euclidean metric. If $A$ is not hyperbolic and $f_{A,v}$ has a fixed point, then the derivative at this fixed point has an eigenvalue on the unit circle, which is impossible for an Anosov fixed point. When the map has no fixed point, it is characterized by $v\notin \operatorname{Im}(I-A)$. A left $1$-eigenvector gives a height coordinate with nonzero drift, and a smooth moving frame reduces the transverse linear map to its orientation class. \\
In the case of item 2, the map admits complete Anosov metrics with every stable index in $\{1,\ldots,n-1\}$; see Theorem~\ref{thm:index}. Metrics realizing different indices are not uniformly equivalent.

Our second main result describes the local geometry of the \emph{Anosov-realizable} locus in the affine parameter space.
    Let $\mathcal{P}_n:=GL(n, \mathbb{R})\times \mathbb R^n$ be the parameter space of affine diffeomorphisms of $\mathbb R^n$. Define the Anosov-realizable locus by
    $$\mathcal{A}_n :=\{(A, v)\in \mathcal{P}_n: f_{A, v} \textrm{ is Anosov with respect to some complete Riemannian metric}\},$$
    and the hyperbolic locus by
    $$\mathcal{H}_n:=\{(A, v)\in \mathcal{P}_n: A \textrm{ is hyperbolic}\}.$$
    Note that Theorem \ref{thm:mainresult} states that $\mathcal{A}_n\setminus \mathcal{H}_n\neq \emptyset$. We use the following terminology. Let $\mathbb{S}^1:=\{z\in \mathbb{C}:|z|=1\}$. A pair $(A, v)\in \mathcal{P}_n$ is called a \emph{regular drift pair} if: 
    \begin{enumerate}
        \item $1$ is the only eigenvalue of $A$ on the unit circle, that is, $\operatorname{Spec}(A)\cap \mathbb{S}^1=\{1\}$;
        \item The eigenvalue $1$ is algebraically simple;
        \item $v\notin \operatorname{Im}(I-A)$.
    \end{enumerate}
    The set of regular drift pairs is denoted by $\mathcal{D}^{reg}_n$.
    
The condition $v\notin \operatorname{Im}(I-A)$ means that the affine drift class
$$[v]_A\in \operatorname{coker}(I-A)=\mathbb R^n/ \operatorname{Im}(I-A)$$
is nonzero. Equivalently, $f_{A, v}$ has no fixed point.

\begin{theorem}[Affine parameter stratification]\label{Aff Par Str}
The set $\mathcal{A}_n$ is dense in $\mathcal{P}_n$ and 
$$\operatorname{Int}_{\mathcal{P}_n}(\mathcal{A}_n)=\mathcal{H}_n\cup \mathcal{D}^{reg}_n.$$
Equivalently, a nonhyperbolic Anosov-realizable affine map is an interior point of $\mathcal{A}_n$ if and only if $1$ is an algebraically simple eigenvalue, $1$ is the only eigenvalue of $A$ on the unit circle, and the map has no fixed point. Therefore
$$\partial_{\mathcal{P}_n}\mathcal{A}_n=\mathcal{P}_n\setminus (\mathcal{H}_n \cup \mathcal{D}^{reg}_n).$$
\end{theorem}

On a compact manifold, the set of Anosov diffeomorphisms is $C^1$-open \cite{Katok}. On a noncompact manifold, however, the corresponding persistence question depends essentially on the topology used to control perturbations at infinity. Compact-open topologies do not detect dynamical obstructions escaping to infinity, whereas the strong Whitney topology permits perturbation sizes that may vary from point to point. The two topologies give different persistence properties for complete Anosov structures on $\mathbb{R}^n$. \\
Firstly, we define
$$\mathscr{D}_n:=\{f\in \operatorname{Diff}^1(\mathbb R^n): f \textrm{ admits a complete Anosov metric}\}.$$

We consider the two-sided weak $C^1_{\mathrm{loc}}$ and strong
Whitney $C^1$ topologies. These definitions are recalled in
Section~\ref{Stability at Infinity}. Thus, our third result shows that Anosov-realizable is not open in the weak $C^1_{loc}$ topology. 

\begin{theorem}[Weak $C^1_{\mathrm{loc}}$ instability]
\label{thm:weak}
Let $f_{A,v}$ be an affine diffeomorphism of $\mathbb R^n$ which has no fixed point. Then $f_{A,v}$ is a boundary point of $\mathscr D_n$ in the two-sided weak $C^1_{\mathrm{loc}}$ topology.

More precisely, there exist smooth diffeomorphisms $h_\varepsilon\notin\mathscr D_n$ such that
$$ h_\varepsilon\longrightarrow f_{A,v} $$
together with their inverses in $C^\infty$ on compact subsets as
$\varepsilon\to 0$.
\end{theorem}

For the strong Whitney topology, the situation is different: $\mathscr D_n$ is open.
\begin{theorem}[Openness in the two-sided strong Whitney $C^1$ topology]\label{thm:whit}
Let $f\in \operatorname{Diff}^1(\mathbb R^n)$ be Anosov with respect to a complete Riemannian metric, and let $k=\dim E^s$. Then there exists a complete smooth Riemannian metric $\hat{g}$ and a two-sided strong Whitney $C^1$ neighborhood $\mathcal{U}$ of $f$ such that every $h\in \mathcal{U}$ is Anosov with respect to $\hat{g}$. Moreover,
$$\dim E^s_h=k, \qquad \dim E^u_h=n-k.$$
The Anosov constants may be chosen independently of $h\in \mathcal U$. Consequently, the set of diffeomorphisms of $\mathbb R^n$ admitting a complete Anosov metric is open in the two-sided strong Whitney $C^1$ topology.

\end{theorem}

\textbf{Structure of the paper.} Section \ref{Prel} collects preliminary definitions and lemmas. In Section \ref{Smooth Conjugacy}, we prove the finite-difference lemma, develop the moving-frame construction, and establish smooth conjugacy to the normal forms; we then prove Theorem \ref{thm:mainresult} and present various examples arising from the classification theorem. Section \ref{index spectrum} studies the index spectrum of affine Anosov maps and proves rigidity results for complete Anosov metrics. Section \ref{Affine parameter geometry} describes the affine parameter geometry and proves \ref{Aff Par Str}. Finally, Section \ref{Stability at Infinity} studies stability at infinity and proves Theorems \ref{thm:weak} and \ref{thm:whit}.

\section{Preliminaries}\label{Prel}
\begin{definition}
Let $M$ be a smooth manifold and $f\colon M\to M$ a diffeomorphism. A diffeomorphism $f:M\to M$ is called \textit{Anosov-realizable} if there exists a complete Riemannian metric $\rho$ on $M$ with respect to which $f$ is Anosov; that is, there exists a continuous $Df$-invariant splitting $TM=E^s\oplus E^u$ and constants $C\ge1$, $\lambda\in(0,1)$ such that for all $n\ge0$:
\[
\|Df^n(v)\|_\rho \le C\lambda^n\|v\|_\rho \quad \forall v\in E^s,\qquad
\|Df^{-n}(w)\|_\rho \le C\lambda^n\|w\|_\rho \quad \forall w\in E^u.
\]
The splitting may be trivial, i.e., $E^s=\{0\}$ or $E^u=\{0\}$ is allowed.
\end{definition}

An affine map on $\mathbb{R}^n$ is $f(x)=Ax+v$ with $A\in GL(n,\mathbb{R})$ and $v\in\mathbb{R}^n$; its derivative is constant: $Df=A$.

\begin{definition}
A matrix $A$ is \textbf{hyperbolic} if it has no eigenvalue on the unit circle $\mathbb{S}^1$.
\end{definition}
\begin{lemma}[Hyperbolic case]\label{lem:hyperbolic}
    If $A$ is hyperbolic, then every affine map $x\mapsto Ax+v$ is Anosov with respect to an adapted Euclidean metric.
\end{lemma}
\begin{proof}
    Let $E_A^s$ and $E_A^u$ be the stable and unstable generalized eigenspaces. Choose an adapted norm so that $\|A|_{E_A^s}\|<1$ and $\|A^{-1}|_{E_A^u}\|<1$. Since $Df=A$ is constant, the constant splitting $E_A^s\oplus E_A^u$ gives the Anosov estimates. The metric is complete. 
\end{proof}

\begin{lemma}[Completeness of weighted product metrics]\label{lem:completeness}
    Let $(B, g_B)$ be a complete Riemannian manifold and let
    $\phi_1, \cdots, \phi_r: B\to (0, \infty)$ be smooth functions. Then $g=g_B+\sum^r_{i=1}\phi_i(b)^2 \textrm{d}\xi^2_i$ defines a complete Riemannian metric on $B\times \mathbb R^r$.
\end{lemma}
\begin{proof}
    Let $ p_m=(b_m,\xi_m) $ be a $g$-Cauchy sequence (Cauchy with respect the metric $g$). Since the projection 
    $ \pi:(B\times\mathbb R^r,g)\longrightarrow(B,g_B) $
is $1$-Lipschitz, $(b_m)$ is a $g_B$-Cauchy sequence. Since $g_B$ is
complete, there exists $b_\infty\in B$ such that
$
b_m\longrightarrow b_\infty.
$

Fix $R>0$. By the Hopf-Rinow theorem, the closed ball
$
K:=\overline{B_{g_B}(b_\infty,R)}
$
is compact. Since each $\phi_i$ is continuous and strictly positive,
$
c_i:=\min_{b\in K}\phi_i(b)>0.
$
Set
$
c:=\min_{1\le i\le r}c_i>0.
$

We show that each fiber coordinate sequence $(\xi_{m,i})$ is Cauchy.
Fix $\eta>0$. Since $(p_m)$ is $g$-Cauchy and $b_m\to b_\infty$, for
all sufficiently large $m,\ell$ we have
$
b_m,b_\ell\in B_{g_B}(b_\infty,R/4)
$
and
$$
d_g(p_m,p_\ell)<
\min\left\{\frac R4,\frac{c\eta}{2}\right\}.
$$
Choose a piecewise smooth curve $\gamma_{m,\ell}$ joining $p_m$ to
$p_\ell$ such that
$$
L_g(\gamma_{m,\ell})
<
d_g(p_m,p_\ell)+\frac{c\eta}{2}<
\min\left\{\frac R2,c\eta\right\}.
$$
The projected curve $\pi\circ\gamma_{m,\ell}$ has $g_B$-length no
greater than $L_g(\gamma_{m,\ell})$. Since its initial point lies in
$B_{g_B}(b_\infty,R/4)$, the whole projected curve remains in $K$.
Therefore $\phi_i\ge c_i$ along $\gamma_{m,\ell}$, and hence
$$
|\xi_{m,i}-\xi_{\ell,i}|
\le
c_i^{-1}L_g(\gamma_{m,\ell})
\le
c^{-1}L_g(\gamma_{m,\ell})
<
\eta.
$$
Thus $(\xi_{m,i})$ is Cauchy for every $i$. Hence $\xi_m\longrightarrow\xi_\infty$
for some $\xi_\infty\in\mathbb R^r$.

It follows that
$
p_m\longrightarrow(b_\infty,\xi_\infty)
$
in the manifold topology. Since the Riemannian distance induces the
manifold topology, the convergence also holds with respect to $d_g$.
Therefore $g$ is complete.

\end{proof}

\section{Smooth Conjugacy to Normal Forms}\label{Smooth Conjugacy}
\begin{lemma}\label{Cohomologous - affine}
    Let $K$ be a finite-dimensional real vector space, let $B\in GL(K)$, and let $q:\mathbb{R}\mapsto K$ be an affine map. Then there exists a polynomial map $\phi: \mathbb{R}\mapsto K$ such that
    $$\phi(t+1)=B\phi(t)+q(t).$$
\end{lemma}
\begin{proof}
    Let $q(t)=at+b$ with $a, b\in K$. We first separate the generalized $1$-eigenvalue of $B$. Let $K_1=\ker (I-B)^r$ for $r$ sufficiently large, and choose a $B$-invariant complement $K_0$ such that $K=K_1\oplus K_0$, where $1$ is not an eigenvalue of $B|_{K_0}$. Decompose $q=q_1+q_0$ according to this direct sum. It is enough to solve the equation on $K_1$ and $K_0$ separately.

    On $K_0$, the operator $I-B$ is invertible. We find an affine solution 
    $$\phi_0(t+1)=B\phi_0(t)+q_0(t).$$
    and comparing the coefficient of $t$ and the constant term gives 
    $$(I-B)\alpha=a_0\qquad (I-B)\beta=b_0-\alpha,  $$
    where $q_0(t)=a_0t+b_0$. Since $I-B$ is invertible on $K_0$, such $\alpha, \beta\in K_0$ exist.

    It remains to solve the equation on $K_1$. On $K_1$, write $B=I+N$ where $N$ is a nilpotent, say $N^r=0$. The equation becomes
    $$\phi_1(t+1)-\phi_1(t)=N\phi_1(t)+q_1(t)\textrm{ or } (\Delta-N)\phi_1=q_1,$$
    where $\Delta \psi(t)=\psi(t+1)-\psi(t).$

    The finite difference operator $\Delta$ is surjective on scalar polynomials. Indeed, $$\Delta\left(\frac{t^{d+1}}{d+1}\right)=t^d+\text{terms of lower degree},$$
    and induction on the degree gives the claim. The same conclusion holds componentwise for $K_1$-valued polynomials. Choose a linear right inverse $S$ of $\Delta$. We take $S$ componentwise, so that it commutes with $N$.

    Define $\phi_1= \sum_{j=0}^{r-1} S^{j+1}N^j q_1,$ which is a polynomial map. Since $N^r=0$, we compute 
    $$\begin{aligned}(\Delta-N)\phi_1 &= \sum_{j=0}^{r-1} S^jN^j q_1-\sum_{j=0}^{r-1} S^{j+1}N^{j+1}q_1\\
    &= q_1-S^rN^r q_1\\
    &=q_1.
    \end{aligned}
    $$
    Thus $\phi_1$ solves the equation on $K_1$.

    Finally, $\phi=\phi_0+\phi_1$ is a polynomial solution of 
    $$\phi(t+1)=B\phi(t)+q(t).$$
    \end{proof}

A linear rescaling of the real variable gives the following corollary.

\begin{corollary}\label{Coro1}
    Let $K$ be a finite-dimensional real vector space, let $B\in GL(K)$, and let $q:\mathbb{R}\mapsto K$ be an affine map. For any fixed $\beta \in \mathbb{R}\setminus\{0\}$, there exists a polynomial map $\phi: \mathbb{R}\mapsto K$ such that
    $$\phi(t+\beta)=B\phi(t)+q(t).$$
\end{corollary}

\begin{proof}
Let $Q(s)=q(\beta s)$. By Lemma~\ref{Cohomologous - affine}, there exists a polynomial map $\psi:\mathbb R\to K$ satisfying
$$\psi(s+1)=B\psi(s)+Q(s).$$
Define $\phi(t)=\psi(t/\beta)$. Then
$$\phi(t+\beta)=\psi(t/\beta+1)=B\psi(t/\beta)+Q(t/\beta)
=B\phi(t)+q(t).$$
Thus $\phi$ has the required property.
\end{proof}

\begin{proposition}[Smooth normal forms]\label{pro:smoconju}
    Let $f_{A, v}(x)=Ax+v$ be an affine diffeomorphism of $\mathbb R^n$. If $n\geq 2$ and $v\notin \operatorname{Im}(I-A)$, then $f_{A, v}$ is smoothly conjugate to one of the following normal forms:
    \begin{itemize}
        \item $(t, z)\mapsto (t+1, z), \qquad \text{if } \det A>0,$
        \item $(t, y, z') \mapsto (t+1, -y, z'), \qquad \text{if } \det A<0$.
    \end{itemize}
\end{proposition}
\begin{proof}

    Since $v\notin \operatorname{Im}(I-A)$, its orthogonal projection onto $\operatorname{Im}(I-A)^{\perp}$ is nonzero. Choose $$e=\frac{\operatorname{proj}_{\operatorname{Im}(I-A)^\perp}v}{\left\|\operatorname{proj}_{\operatorname{Im}(I-A)^\perp}v\right\|},$$
    and define $\ell(x)=\langle x, e\rangle$. Then $\ell(e)=1$ and $$\ell(v)=\left\|\operatorname{proj}_{\operatorname{Im}(I-A)^\perp}v\right\|>0.$$
    Moreover, $e\perp\operatorname{Im}(I-A)$, so $\ell((I-A)x)=0$ for every $x\in \mathbb{R}^n$ and hence $\ell\circ A=\ell$. Let $t=\ell(x)$ and $K=\ker\ell$. Since $\ell\circ A=\ell$, the subspace $K$ is $A$-invariant.

    Every $x$ can be written uniquely as $x=te+z$ where $z\in K$. Since $\ell(f(x))=\ell(x)+\ell(v)$, the first coordinate of $f$ is $t+\ell(v)$. In these coordinates $f$ has the following form:\\
    
    \noindent \textbf{Claim:} In the coordinates $\mathbb R e\oplus K$, $f$ has the form
    $$f(t, z)=(t+\ell(v), Bz+q(t)),$$
    where $B=A|_K\in GL(K)$ is an invertible linear map of $K$, and $q(t)$ is an affine $K$-valued function.
    \begin{proof}[\emph{\textbf{Proof of Claim}}]
    Write $x=te+z$. Then
    $$f(x)=A(te+z)+v=tAe+Az+v.$$
    Subtracting $(t+\ell(v))e$, the $K$-component is
    $$Bz+t(Ae-e)+(v-\ell(v)e).$$
    Both $Ae-e$ and $v-\ell(v)e$ are in $K$ since applying $\ell$ gives zero. Hence
    $$q(t)=t(Ae-e)+(v-\ell(v)e).$$
    Thus $q(t)$ is affine.
    \end{proof}
    With respect to the decomposition
    $$ \mathbb R^n=\mathbb Re\oplus K,$$
    the matrix of $A$ has the form $$A=
   \begin{pmatrix}
    1&0\\
    a&B
    \end{pmatrix},$$
    where $a=Ae-e$ and $B=A|_K$. In particular, $\det A=\det B$.
    Now we remove this affine fiber translation. From Corollary \ref{Coro1}, we choose a smooth function $\phi: \mathbb R\to K$ satisfying the difference equation
    \begin{equation}\label{Eq1.Prop 1.6}
    \phi(t+\ell(v))=B\phi(t)+q(t).    
    \end{equation}
    Define
    $$H_1(t, z)=(t, z-\phi(t)).$$
    Then $H_1$ is a global smooth diffeomorphism with inverse
    $$H^{-1}_1(t, z)=(t, z+\phi(t)).$$
    Equation~\ref{Eq1.Prop 1.6} gives 
    $$H_1\circ f\circ H^{-1}_1(t, z)=(t+\ell(v), Bz).$$
    Thus $f$ is smoothly conjugate to 
    $$L(t, z)=(t+\ell(v), Bz).$$
    Note that, as $\ell(v)\neq 0$, the map $L$ is smoothly conjugate to $\tilde{L}(t,z)=(t+1,Bz)$ through the map $\tilde{H}(t,z)=\bigl(t/\ell(v),\,z\bigr)$. After this rescaling we can assume from now on that $\ell(v)=1$.
    
    It remains to simplify the matrix $B$. Let $m=n-1$.  Since $n\geq 2$, $m\geq 1$. Define $$R=\begin{cases}
        I_m &  \textrm{ if } \det A>0,\\
        \operatorname{diag}(-1, 1, \cdots, 1)& \textrm{ if } \det A<0.
    \end{cases}$$
    Since $\det A=\det B$, this is the same as choosing $R$ according to the sign of $\det B$. In either case, $\det(RB^{-1})>0$. Hence $RB^{-1}\in GL^+(m, \mathbb{R})$, where $GL^+(m, \mathbb{R})$ denotes the connected component of $GL(m, \mathbb{R})$ consisting of matrices with positive determinant. Since $GL^+(m, \mathbb R)$ is path connected, there exists a smooth path $C:[0, 1]\to GL^+(m, \mathbb R)$ such that
    $$C(0)=I_m, \qquad C(1)=RB^{-1}.$$
    Choose a smooth non-decreasing function $\eta:[0,1]\to[0,1]$ which is equal to $0$ near $0$ and to $1$ near $1$. After replacing $C$ by $C\circ\eta$, we may assume that $C$ is constant near both endpoints. Extend $C$ to $\mathbb R$ by
    $$ C(t+1)=RC(t)B^{-1}.$$
    Since $C$ is constant near the endpoints of $[0,1]$, this extension is smooth. Define
    $$H_2(t,z)=(t,C(t)z).$$
    Since $C(t)\in GL(m,\mathbb R)$, this is a global smooth diffeomorphism with inverse $$H_2^{-1}(t,z)=(t,C(t)^{-1}z).$$
    We have $$H_2\circ L(t,z)=(t+1,C(t+1)Bz).$$
    By definition of $C$, $C(t+1)B=RC(t)$. Therefore,
    $$H_2\circ L(t, z)=(t+1, RC(t)z).$$
    On the other hand, if $L_R(t, z)=(t+1, Rz),$ then 
    $$L_R(H_2(t, z))=L_R(t, C(t)z)=(t+1, RC(t)z).$$
    Thus 
    $$H_2\circ L=L_R \circ H_2.$$
    Hence $L$ is smoothly conjugate to $L_R$. Since $\det A=\det B$, if $\det A>0$, then $R=I_m$ and $L_R(t,z)=(t+1, z)$ is a pure translation. If $\det A<0$, then $R=\operatorname{diag}(-1, 1,\cdots, 1)$ and $L_R(t, y, z')=(t+1, -y, z')$, which is the White's map times identity on $\mathbb R^{n-2}$.
    
    \end{proof}
\subsection{Proof of Main Result}
    
\begin{lemma}[Smooth conjugacy and complete Anosov metrics]
\label{lem:smopre}
Let $H:M\to N$ be a $C^1$ diffeomorphism and suppose that
$$
G=H\circ F\circ H^{-1}.
$$
Then $F$ is Anosov-realizable if and only if $G$ is
Anosov-realizable. More precisely, if $\rho$ is a complete Anosov
metric for $G$, then $H^*\rho$ is a complete Anosov metric for $F$.
\end{lemma}
\begin{proof}
Define $H^*\rho$ by
$$(H^*\rho)_x (u, \omega)=\rho_{H(x)}(DH_x u, DH_x \omega).$$
Then $H:(M, H^*\rho)\to (N, \rho)$ is an isometry. The stable bundle $E^s_G$ and unstable bundle $E^u_G$  of $G$ pull back under $DH^{-1}$ to stable bundle $E^s_F$ and unstable bundle $E^u_F$ of $F$.
$$E^s_F(x)=DH^{-1}_{H(x)}(E^s_G(H(x))),$$
$$E^u_F(x)=DH^{-1}_{H(x)}(E^u_G(H(x))),$$
Since $G\circ H =H\circ F$, we have 
$$DG_{H(x)}\circ DH_x=DH_{F(x)}\circ DF_x,$$
which means that $E^s_F, E^u_F$ are $DF$-invariant.

Since $H$ is an isometry, the Anosov contraction gives
$$\|DF^n_x u\|_{H^*\rho, F^n(x)}=\|DG^n_{H(x)}(DH_x u)\|_{\rho, G^n(H(x))}\leq C \lambda^n \|DH_x u\|_{\rho, H(x)}=C\lambda^n\|u\|_{H^*\rho, x}.$$
The corresponding estimate on the unstable bundle follows in the
same way from the estimate for $G^{-1}$.

Also if $(N, \rho)$ is complete, then $(M, H^*\rho)$ is complete.
The converse follows by applying the same argument to $H^{-1}$.
\end{proof}

\begin{lemma}[White's model maps]\label{lem:white'smap}
For every $n\geq 2$, the maps
$$T_+(t,z)=(t+1,z) \textrm{ and } T_-(t,y,z')=(t+1,-y,z').$$
on $\mathbb R^n$ admit complete Riemannian metrics with respect to which they are Anosov.
\end{lemma}
\begin{proof}
White constructed a complete Riemannian metric $g_W$ on $\mathbb R^2$ for which $W(t, y)=(t+1, -y)$ is Anosov. Hence $W^2(t, y)=(t+2, y)$ is Anosov with respect to the same metric. Since the affine map $S(t, y)=(2t, y)$ satisfies
   $$S\circ (t,y\mapsto(t+1,y))=W^2\circ S, $$
   the two-dimensional translation $(t, y)\mapsto (t+1, y)$ also admits a complete Anosov metric, which we denote $g_{+, 2}$.

   For $n>2$, write the remaining coordinates as $\xi \in \mathbb R^{n-2}$. Define  
   $$g_- = g_W(t,y)+\sum_{j=1}^{n-2} e^{2t}\,d\xi_j^2$$
   for $T_-$ and $$ g_+ = g_{+,2}(t,y)+\sum_{j=1}^{n-2} e^{2t}\,d\xi_j^2$$
   for $T_+$. By Lemma~\ref{lem:completeness}, $g_+$ and $g_-$ are complete. For the Anosov estimates, use the stable and unstable bundles of the two-dimensional base and place each additional direction $\partial_{\xi_j}$ in the unstable bundle.
   Under either $T_+$ or $T_-$, the $t$-coordinate increases by $1$, and therefore
   $$\|D T_\pm(\partial_{\xi_j})\|_{g_\pm,T_\pm(t,\cdot)}=e\,\|\partial_{\xi_j}\|_{g_\pm,(t,\cdot)}. $$
   This implies that the inverse contracts these fiber directions by the factor $e^{-1}$. Combining with the Anosov estimates on the two-dimensional base, this gives an invariant splitting and uniform contraction and expansion estimate for $T_+$ and $T_-$. Hence both maps are Anosov with respect to complete Riemannian metrics.
\end{proof}

\begin{proof}[\textbf{\emph{Proof of Theorem~\ref{thm:mainresult}}}]
For necessity, suppose that $f_{A,v}$ is Anosov with respect to a
complete Riemannian metric. If $A$ is not hyperbolic, then $A$ has an eigenvalue of modulus one. Suppose, by contradiction, that
$p$ is a fixed point. Since $Df_p=A$, the Anosov splitting at $p$
is an $A$-invariant decomposition
$$T_p\mathbb R^n=E_p^s\oplus E_p^u.$$
On the stable subspace $E^s_p$, the Anosov estimate gives exponential contraction. Therefore every eigenvalue of $Df_p|_{E^s_p}$ has modulus less than 1. On the unstable subspace $E^u_p$, the inverse estimate gives exponential contraction for $Df_p^{-1}$. Therefore every eigenvalue of $Df_p|_{E^u_p}$ has modulus greater than 1. Hence $Df_p$ cannot have any eigenvalue of modulus 1. But $Df_p=A$ has an eigenvalue of modulus one, a contradiction. Therefore $f_{A,v}$ has no fixed point. Equivalently,
$$v\notin\operatorname{Im}(I-A).$$

For sufficiency, suppose first that $A$ is hyperbolic.
Lemma~\ref{lem:hyperbolic} gives a complete adapted Euclidean metric for which $f_{A,v}$ is Anosov.

Now suppose that $v\notin\operatorname{Im}(I-A)$.
By Proposition~\ref{pro:smoconju}, $f_{A,v}$ is smoothly conjugate to one of the model maps $T_+$ or $T_-$. By
Lemma~\ref{lem:white'smap}, both model maps admit complete Anosov
metrics. Lemma~\ref{lem:smopre} then gives a complete Anosov metric for $f_{A,v}$.
\end{proof}
\begin{theorem}[Smooth conjugacy classification]\label{thm:smoonormal}
    Let $f(x)=Ax+v$ be an affine diffeomorphism of $\mathbb R^n$, where $A\in GL(n, \mathbb R)$ and $n\ge 2$. Suppose that $f_{A,v}$ is Anosov-realizable. Then one of the following happens.
    \begin{itemize}
        \item $A$ is hyperbolic. In this case, $f_{A, v}$ is affine conjugate to the hyperbolic linear map $x \mapsto Ax$.
        \item $A$ is not hyperbolic, $v\notin \operatorname{Im}(I-A)$  and $\det A>0$. In this case, $f_{A, v}$ is smoothly conjugate to
        $$T_+(t, z)=(t+1, z).$$
        \item $A$ is not hyperbolic, $v\notin\operatorname{Im}(I-A)$, and  $\det A<0$. In this case, $f_{A, v}$ is smoothly conjugate to 
        $$T_-(t, y, z')=(t+1, -y, z').$$
    \end{itemize}
    Conversely, every map in one of these three cases is Anosov-realizable.
\end{theorem}
\begin{proof}
   If $A$ is hyperbolic, then $1 \notin \operatorname{Spec}(A)$, so $I-A$ is invertible. Hence $f_{A, v}$ has a unique fixed point $p$ such that $(I-A)p=v$. The translation $H(x)=x+p$ gives
    $$H^{-1}\circ f_{A,v}\circ H(x)=H^{-1}(A(x+p)+v)=Ax+Ap+v-p=Ax.$$
    Thus $f_{A, v}$ is affine conjugate to the hyperbolic linear map $x\mapsto Ax$.

    If $A$ is not hyperbolic, Theorem~\ref{thm:mainresult} gives $v\notin\operatorname{Im}(I-A).$ Proposition~\ref{pro:smoconju} gives the remaining cases.

    Conversely, if $A$ is hyperbolic, then the linear map $x\mapsto Ax$ is Anosov with respect to an adapted complete Euclidean metric by Lemma~\ref{lem:hyperbolic}. The $T_-$ and $T_+$ admit complete Anosov metrics by Lemma~\ref{lem:white'smap}. Finally, any smooth conjugate of one of these maps is Anosov with respect to the pulled-back complete metric by Lemma~\ref{lem:smopre}. Thus we prove the reverse.
\end{proof}

We give three examples illustrating the classification.

\begin{example}[Real diagonal case]\label{ex:real-diagonal}
Let $$A=\operatorname{diag}(\lambda_1,\ldots,\lambda_n)\in GL(n,\mathbb R),
\qquad
v=(v_1,\ldots,v_n).$$
Then $f_{A, v}(x)$ is Anosov-realizable in either of the following cases:
\begin{enumerate}
\item $\left|\lambda_i\right|\neq1$ for all $i$. Then $A$ is hyperbolic, so Lemma~\ref{lem:hyperbolic} applies.

\item There exists an index $i$ such that
$$\lambda_i=1,\qquad v_i\neq0. $$ 
If $w=(I-A)x$, then $w_i=(1-\lambda_i)x_i=0.$ Thus every vector in $\operatorname{Im}(I-A)$ has zero $i$-th coordinate, whereas $v_i\neq0$. Hence
$$v\notin\operatorname{Im}(I-A).$$
By Theorem~\ref{thm:mainresult}, $f_{A,v}$ admits a complete Anosov metric.
\end{enumerate}
Moreover, in the second case, after permuting coordinates we may assume $i=1$. Then
$$A=1\oplus B,\qquad
B=\operatorname{diag}(\lambda_2,\ldots,\lambda_n), $$
and $v=(c,u)$ with $c\neq 0$. By Proposition~\ref{pro:smoconju}, the map is smoothly conjugate to a pure translation if $\det A>0$, and to White's map times the identity if $\det A<0$.
\end{example}

\begin{example}[A family with no fixed point]
Let $$A=
\begin{pmatrix}
1 & 0\\
a & B
\end{pmatrix},
\qquad
a\in\mathbb R^{n-1},\quad B\in GL(n-1,\mathbb R),$$
and let
$$v=(\tau,w),\qquad \tau\neq 0. $$
Then $(I-A)(t, z)=\bigl(0,-at+(I-B)z\bigr)$. Hence $\operatorname{Im}(I-A)\subset \{0\}\times\mathbb R^{n-1}$. Since the first coordinate of $v$ is $\tau\neq 0$, we have $v\notin\operatorname{Im}(I-A).$ Therefore $f_{A, v}$ has no fixed point and is Anosov-realizable by Theorem~\ref{thm:mainresult}.

Moreover, since $\det A=\det B,$ Proposition~\ref{pro:smoconju} shows that $f_{A,v}$ is smoothly conjugate to a pure translation if $\det B>0$, and to White's map times the identity if $\det B<0$.
\end{example}

\begin{example}[Unit complex block with drift]\label{ex:complex-drift}
Let    $$A=
\begin{pmatrix}
1&0&0\\
0&a&-b\\
0&b&a
\end{pmatrix},
\qquad a^2+b^2=1,\quad b\neq0, $$ and let $v=(\tau,u_1,u_2),\qquad \tau\neq 0.$ 

The eigenvalues of $A$ are $1$ and $a\pm ib$, all of which lie on the unit circle. Since the first coordinate of every vector in
$\operatorname{Im}(I-A)$ is zero whereas $\tau\ne0$, we have
$v\notin\operatorname{Im}(I-A)$. Hence $f_{A,v}$ has no fixed point and is Anosov-realizable.

In the special case $\tau=1$ and $(u_1, u_2)=(0, 0)$, write $A=1\oplus R_\theta$. Then $L(t,z)=(t+1,R_\theta z)$ is conjugate to the pure translation $(t,z)\mapsto(t+1,z)$ via
$$
H(t,z)=(t,R_{-t\theta}z).
$$
Indeed,
$$H\circ L(t,z)=(t+1,R_{-(t+1)\theta}R_\theta z)=(t+1,R_{-t\theta}z). $$
\end{example}

\section{The index spectrum of affine Anosov maps}\label{index spectrum}

\begin{definition}[Index Spectrum]
Let $f_{A, v}(x)=Ax+v$ be an affine diffeomorphism of $\mathbb R^n$, $n\ge 2$, which admits a complete Anosov metric. For such $f_{A, v}$ define
    $$\operatorname{Ind}(f_{A, v})=\{\dim E^s_g:
g \text{ is a complete Riemannian metric with respect to which } f_{A, v} \text{ is Anosov}\},$$
where $E^s_g$ denotes the stable bundle of an Anosov splitting associated with $g$ and the set is taken over all such metrics and splittings. Since $\mathbb R^n$ is connected and the splitting is continuous, $\dim E^s_g(x)$ is independent of $x$.
\end{definition}

\begin{theorem}[Index Spectrum Theorem]\label{thm:index}
    Let $f_{A, v}(x)=Ax+v$ be an affine diffeomorphism of $\mathbb R^n$, $n\ge 2$, which admits a complete Anosov metric. Then the following hold:
\begin{enumerate}
    \item[\emph{\text{1.}}] If $A$ is hyperbolic, then $\operatorname{Ind}(f_{A, v})=\{\dim E^s_A\}$, where $E^s_A = \bigoplus_{|\lambda| < 1} E_{\lambda}^{\mathrm{gen}}$ is the stable generalized eigenspace of $A$.
    \item[\emph{\text{2.}}] If $A$ is not hyperbolic and $v\notin \operatorname{Im}(I-A)$, then $\operatorname{Ind}(f_{A, v})=\{1, 2, \dots, n-1\}$.
\end{enumerate}
Moreover, in the second case, for every $k\in \{1, \dots, n-1\}$, there exists a complete Riemannian metric $g_k$ and an Anosov splitting $T\mathbb R^n=E^s_k \oplus E^u_k$ such that 
$$\dim E^s_{k} =k, \qquad \dim E^u_{k} =n-k.$$
The metrics may be chosen so that $g_k$ and $g_j$ are not uniformly equivalent whenever $k\ne j$. In other words, the identity map between $(\mathbb R^n, g_k)$ and $(\mathbb R^n, g_j)$ is not bi-Lipschitz.
\end{theorem}
\begin{proof}
    Suppose that $A$ is hyperbolic. Then $I-A$ is invertible, so $f_{A, v}$ has a unique fixed point $p$ satisfying $(I-A)p=v$. Let $g$ be any complete Anosov metric for $f_{A,v}$, and let $C\ge1$ and $\lambda\in(0,1)$ be corresponding Anosov constants. At the point $p$, the Anosov splitting gives the $A$-invariant decomposition
    $$T_p \mathbb R^n=E^s_g(p)\oplus E^u_g(p).$$
    Let $V^s=E_g^s(p)$ and $V^u=E_g^u(p)$. The stable estimate at the fixed point gives $\left\|(A|_{V^s})^m\right\|\leq C\lambda^m$. Taking $m$-th roots and using the spectral-radius formula yields $r(A|_{V^s})\leq \lambda<1$. Hence every eigenvalue of $A|_{V^s}$ has modulus less than 1. Similarly, $\left\|(A|_{V^u})^{-m}\right\|\leq C\lambda^m$, so $r((A|_{V^u})^{-1})\leq \lambda<1$. Thus every eigenvalue of $A|_{V^u}$ has modulus $>1$. Because $V^s$ and $V^u$ are complementary $A$-invariant subspaces, complexification gives the complementary $A_{\mathbb C}$-invariant decomposition
    $$\mathbb C^n=V^s_{\mathbb C}\oplus V^u_{\mathbb C}.$$
    Therefore $$
    \chi_{A_{\mathbb C}}(z)=\chi_{A_{\mathbb C}|V^s_{\mathbb C}}(z)\,\chi_{A_{\mathbb C}|V^u_{\mathbb C}}(z),$$
   where, for any linear operator $T$ on a finite-dimensional vector space, $$
   \chi_T(z):=\det(zI-T)$$
   denotes its characteristic polynomial.

   The estimates above imply
   $$ \operatorname{Spec}(A_{\mathbb C}|V^s_{\mathbb C})\subset\{|z|<1\},\qquad \operatorname{Spec}
   (A_{\mathbb C}|V^u_{\mathbb C})\subset\{|z|>1\}.$$
   
   Since $A$ is hyperbolic, the first factor contains exactly the
   eigenvalues of $A$ inside the unit disk, counted with algebraic multiplicity. Therefore $$\deg\chi_{A_{\mathbb C}|V^s_{\mathbb C}}=\dim E_A^s.$$
   Since $$ \deg\chi_{A_{\mathbb C}|V^s_{\mathbb C}}=\dim V^s,$$
   we obtain $$\dim V^s=\dim E_A^s.$$
   Since the stable bundle has constant dimension on the connected manifold $\mathbb R^n$, we have $\dim E^s=\dim E^s_A$. The metric and splitting were arbitrary, so 
    $$\operatorname{Ind}(f_{A, v})=\{\dim E^s_A\}.$$
    
    Now suppose that $A$ is not hyperbolic and $v\notin \operatorname{Im}(I-A)$. by Proposition \ref{pro:smoconju}, $f_{A, v}$ is smoothly conjugate to one of 
    $$T_+(t, y, \eta)=(t+1, y, \eta), \textrm{  or  } T_-(t, y, \eta)=(t+1, -y, \eta).$$
    where $\eta\in \mathbb R^{n-2}$. We write these two maps as
    $$T_{\sigma}(t, y, \eta)=(t+1, \sigma y, \eta), \qquad \sigma \in \{1, -1\}.$$
    Let $H$ be a smooth conjugacy satisfying
    $$H \circ f_{A, v}\circ H^{-1} =T_\sigma.$$
    We construct the required metrics for $T_\sigma$ and then pull them back by $H$.

    Let $g_{\sigma,2}$ be a complete metric for which $T_{\sigma,2}$ is Anosov, and let $C_0\ge1$ and $\lambda_0\in(0,1)$ be corresponding Anosov constants. The map $T_{\sigma,2}$ has no periodic points, since its first coordinate increases by $m$ under the $m$-th iterate. We claim that neither invariant bundle can be trivial. Suppose first that $E^s_{\sigma,2}=T\mathbb R^2$.

    Choose $N$ so that $q=C_0\lambda_0^N<1$. Then $\left\|DT^N_{\sigma, 2}v\right\|\leq q\left\| v\right\|$ for every tangent vector. Thus
    $$d(T^N_{\sigma, 2}x, T^N_{\sigma, 2}y)\leq q d(x, y).$$
    Since the metric is complete, Banach's fixed point theorem implies that $T^N_{\sigma, 2}$ has a fixed point, which is a contradiction. Applying the same argument to $T^{-1}_{\sigma, 2}$ excludes the possibility that its unstable bundle is all of $T\mathbb R^2$. Hence both bundles are nontrivial. Since the base has dimension two, we have
    $$\dim E^s_{\sigma, 2}=\dim E^u_{\sigma, 2}=1.$$
    Fix $k\in \{1, \cdots, n-1\}$ and write $\eta=(\eta_1, \cdots, \eta_{n-2})$. Define
    $$g_{\sigma, k}=g_{\sigma, 2}+\sum^{k-1}_{i=1} e^{-2t} d \eta^2_i+\sum^{n-2}_{i=k} e^{2t} d \eta^2_i.$$
    Lemma~\ref{lem:completeness} shows that $g_{\sigma,k}$ is complete. Define
    $$E^s_{\sigma, k}=E^s_{\sigma, 2}\oplus \operatorname{span}\{\partial_{\eta_1}, \dots, \partial_{\eta_{k-1}}\},$$
    and $$E^u_{\sigma, k}=E^u_{\sigma, 2}\oplus \operatorname{span}\{\partial_{\eta_k}, \dots, \partial_{\eta_{n-2}}\},$$
    These bundles are continuous and $DT_{\sigma}$-invariant. For $i<k$, we have
    $$\left\|DT^m_{\sigma}\partial_{\eta_i}\right\|=e^{-m}\left\|\partial_{\eta_i}\right\|.$$
    For $i\geq k$, we have
    $$\left\|DT^{-m}_{\sigma}\partial_{\eta_i}\right\|=e^{-m}\left\|\partial_{\eta_i}\right\|.$$
    Set
    $$C_1=\max\{C_0, 1\}, \qquad \lambda_1=\max\{\lambda_0, e^{-1}\}.$$
    Then $\lambda_1<1$. Since the base and fiber summands are orthogonal,
   $$\left\|DT^m_{\sigma}v\right\|\leq C_1 \lambda^{m}_1\left\|v \right\|, \qquad v\in E^s_{\sigma, k},$$
   and $$\left\|DT^{-m}_{\sigma}w\right\|\leq C_1 \lambda^{m}_1\left\|w \right\|, \qquad w \in E^u_{\sigma, k}.$$
   Thus $T_\sigma$ is Anosov with respect to $g_{\sigma, k}$ and 
   $$\dim E^s_{\sigma, k}=k, \qquad \dim E^u_{\sigma, k}=n-k.$$
   Define $g_k=H^* g_{\sigma, k}$. Then $H:(\mathbb R^n,g_k)\longrightarrow (\mathbb R^n,g_{\sigma,k})$ is an isometry. Hence $g_k$ is complete, and $f_{A, v}$ is Anosov with respect to $g_k$, with stable index k. Therefore every $k\in\{1, \dots, n-1\}$ belongs to $\operatorname{Ind}(f_{A, v})$.

   The normal forms $T_\sigma$ have no periodic points. Since periodic points are preserved under smooth conjugacy, $f_{A,v}$ also has no periodic points. If an Anosov splitting has $E^s=T\mathbb R^n$, the contradiction above implies that some iterate of $f_{A, v}$ is a strict contraction of a complete metric space, and hence has a fixed point. If $E^u=T\mathbb R^n$, the same argument applies to an inverse iterate. Both contradict the absence of periodic points. Hence indices $0$ and $n$ cannot happen. Therefore
   $$\operatorname{Ind}(f_{A, v})=\{1, \dots, n-1\}.$$
   Suppose $k<j$. Then $\partial_{\eta_k}$ is unstable for $g_{\sigma,k}$ and stable for $g_{\sigma,j}$. At $x_t=(t,0,0)$,
   $$\|\partial_{\eta_k}\|_{g_{\sigma,k},x_t}=e^t,\qquad \|\partial_{\eta_k}\|_{g_{\sigma,j},x_t}=e^{-t}.$$
   Hence $$ \frac{\|\partial_{\eta_k}\|_{g_{\sigma,k},x_t}} {\|\partial_{\eta_k}\|_{g_{\sigma,j},x_t}}=e^{2t}
   \longrightarrow\infty \qquad\text{as }t\to\infty.$$
   Thus the identity from $(\mathbb R^n,g_{\sigma,j})$ to
   $(\mathbb R^n,g_{\sigma,k})$ is not Lipschitz. Conversely, $$
   \frac{\|\partial_{\eta_k}\|_{g_{\sigma,j},x_t}} {\|\partial_{\eta_k}\|_{g_{\sigma,k},x_t}}= e^{-2t}\longrightarrow\infty \qquad\text{as }t\to-\infty.$$
   Thus the identity in the opposite direction is not Lipschitz either. Hence $g_{\sigma,k}$ and $g_{\sigma,j}$ are not uniformly equivalent. Since $$g_k=H^*g_{\sigma,k},\qquad g_j=H^*g_{\sigma,j},$$
   the same conclusion holds for $g_k$ and $g_j$.
   \end{proof}

\begin{proposition}[Metric growth on the $1$-eigenspace]
\label{prop:metric-growth}
Let $f_{A,v}(x)=Ax+v$ be an affine diffeomorphism of $\mathbb R^n$ with no fixed point. Let $g$ be a complete Riemannian metric for which $f_{A,v}$ is Anosov, with constants $C\ge1$ and $\lambda\in(0,1)$. Then, for every $0\ne w\in\ker(A-I)$ and every $x\in\mathbb R^n$, there exist $\sigma\in\{-1,1\}$ and $c=c(x, w)>0$ such that
$$ \|w\|_{g,f^{\sigma n}(x)} \ge
c\lambda^{-n}$$
for all sufficiently large $n$.
\end{proposition}

\begin{proof}
Fix $0\ne w\in\ker(A-I)$ and $x\in\mathbb R^n$. Since $Aw=w$,
we have $A^nw=w$ for every $n\in\mathbb Z$. Write
$$ w=w^s+w^u, \qquad w^s\in E^s(x),\quad w^u\in E^u(x).
$$
Suppose first that $w^u\ne0$. The Anosov estimates give
$$ \|A^nw^u\|_{g,f^n(x)} \ge C^{-1}\lambda^{-n}\|w^u\|_{g,x}
$$
and
$$ \|A^nw^s\|_{g,f^n(x)} \le C\lambda^n\|w^s\|_{g,x}.
$$
Since $A^nw=w$,
$$ \|w\|_{g,f^n(x)} \ge C^{-1}\lambda^{-n}\|w^u\|_{g,x} - C\lambda^n\|w^s\|_{g,x}.
$$
For all sufficiently large $n$, the second term is at most one half of the first one. Hence
$$ \|w\|_{g,f^n(x)} \ge \frac{1}{2C}\lambda^{-n}\|w^u\|_{g,x}.
$$
Suppose now that $w^u=0$. Then $w\in E^s(x)$. By invariance of the stable bundle, $w\in E^s(f^{-n}(x))$, and the stable estimate from $f^{-n}(x)$ to $x$ gives 
$$ \|w\|_{g,x} = \|A^nw\|_{g,x} \le C\lambda^n\|w\|_{g,f^{-n}(x)}.
$$
Therefore
$$ \|w\|_{g,f^{-n}(x)} \ge C^{-1}\lambda^{-n}\|w\|_{g,x}
$$
for every $n\ge0$. Thus the required estimate holds along either the forward or the backward orbit.
\end{proof}

\begin{theorem}[Uniformly Euclidean rigidity]
    Let $f_{A, v}(x)=Ax+v$ be an affine diffeomorphism of $\mathbb R^n$. The following are equivalent:
    \begin{enumerate}
        \item[\emph{\text{1.}}] A is hyperbolic.
        \item[\emph{\text{2.}}]  There exists a complete Anosov metric $g$ for $f_{A, v}$ which is uniformly equivalent to the Euclidean metric. That is, there exists $L\geq 1$ such that
        $$L^{-1}|w|\leq \left\| w\right\|_{g, x}\leq L|w|$$
        for every $x\in \mathbb R^n$ and $w\in T_x\mathbb R^n$.
        \item[\emph{\text{3.}}]  There exists a constant Riemannian metric for which $f_{A, v}$ is Anosov.
    \end{enumerate}
\end{theorem}
\begin{proof}
    We prove $1\implies 3$ first. If $A$ is hyperbolic, choose an adapted inner product for the splitting $\mathbb R^n=E^s_A \oplus E^u_A$ such that $A|_{E^s_A}$ and $A^{-1}|_{E^u_A}$ are strict contractions. The associated constant metric is complete and $f_{A, v}$ is Anosov. 
    
    The implication of $3\implies 2$ follows since any two constant norms on a finite-dimensional vector space are uniformly equivalent. 

    It remains to prove $2\Longrightarrow1$. Suppose that $g$ is uniformly equivalent to the Euclidean metric and makes $f_{A,v}$ Anosov. Let $C\ge1$ and $\lambda\in(0,1)$ be corresponding Anosov constants. Assume, by contradiction, that $A$ is not hyperbolic. By Theorem~\ref{thm:mainresult}, $f_{A,v}$ then has no fixed point. In particular, $I-A$ is not invertible, so $1\in\operatorname{Spec}(A)$. Choose $0\ne w\in\ker(A-I)$.

    Fix $x\in\mathbb R^n$. By Proposition~\ref{prop:metric-growth}, there exist $\sigma\in\{-1,1\}$ and $c>0$ such that $$ \|w\|_{g,f^{\sigma n}(x)}\ge c\lambda^{-n}$$
    for all sufficiently large $n$. On the other hand, uniform equivalence with the Euclidean metric gives $$\|w\|_{g,f^{\sigma n}(x)}\le L|w| $$
    for every $n$, since $w$ is the same Euclidean vector at every base point. This is a contradiction. Therefore $A$ is hyperbolic.
\end{proof}

\section{Affine parameter geometry}\label{Affine parameter geometry}
        
\begin{lemma}[Spectrum perturbation]\label{lem:specpert}
    Let $A\in GL(n, \mathbb{R})$ and suppose that $1\in \operatorname{Spec}(A)$. Assume that at least one of the following holds:
    \begin{enumerate}
        \item[\emph{1}.] $A$ has an eigenvalue in $\mathbb{S}^1\setminus \{1\}$.
        \item[\emph{2}.] The algebraic multiplicity of the eigenvalue 1 is at least two.
    \end{enumerate}
    Then every neighborhood of $A$ in $GL(n, \mathbb{R})$ contains a real matrix $A'$ such that
    $$1\notin \operatorname{Spec}(A'), \qquad \operatorname{Spec}(A')\cap \mathbb{S}^1\neq \emptyset.$$
\end{lemma}
\begin{proof}
     Assume that $\mu\in\operatorname{Spec}(A)\cap \bigl(\mathbb S^1\setminus\{1\}\bigr).$ Let $V_1$ be the generalized eigenspace associated with the eigenvalue $1$, and choose an $A$-invariant complement $W$. Thus $\mathbb R^n=V_1\oplus W$. Set $A_1:=A|_{V_1}$. Since $\mu \neq 1$, we have $\mu \in \operatorname{Spec}(A|_W)$. For a sufficiently small nonzero real number $\varepsilon$, define $A_{\varepsilon}$ by
    $$A_{\varepsilon}|_{V_1}=A_1+\varepsilon I, \qquad A_\varepsilon|_W=A|_W.$$
    Then $A_\varepsilon\to A$ as $\varepsilon\to 0$. Since $\operatorname{Spec}(A_1+\varepsilon I)= \{1+\varepsilon\}$,
    we have $1\notin \operatorname{Spec}(A_\varepsilon)$. On the other hand, the restriction to $W$ is unchanged, so $\mu \in \operatorname{Spec}(A_\varepsilon)\cap \mathbb{S}^1$. This proves the result in the first case.

    Now suppose that $1$ is the only eigenvalue of $A$ on the unit circle and that its algebraic multiplicity is at least two. Put the restriction of $A$ to its generalized $1$-eigenspace in real Jordan form.
    
    Suppose first that this Jordan form contains a block $J_r(1)$ with $r\ge2$. Write this block in the form
$$ J_r(1) =\begin{pmatrix}
J_2(1) & C\\
0 & J_{r-2}(1)
\end{pmatrix},
$$
where the lower-right block is omitted when $r=2$. Replace $J_2(1)$ by $$M_\theta =
\begin{pmatrix}
1&1\\
-2(1-\cos\theta)&2\cos\theta-1
\end{pmatrix}.
$$
Then $ M_\theta\longrightarrow J_2(1)$ as $\theta\to 0$. Moreover,
$ \operatorname{tr}M_\theta=2\cos\theta$ and $\det M_\theta=1. $
Hence the characteristic polynomial of $M_\theta$ is $z^2-2\cos\theta\,z+1,$
whose roots are $e^{i\theta}$ and $e^{-i\theta}$.

If $r>2$, replace the remaining diagonal block $J_{r-2}(1)$ by
$J_{r-2}(1+\delta)$, where $\delta\neq0$ is arbitrarily small, while
leaving the upper-right block $C$ unchanged. Replace every other Jordan block $J_s(1)$ by $J_s(1+\delta_s)$ with $\delta_s\neq 0$ arbitrarily small. The resulting matrix is block upper triangular, so its spectrum is the union of the spectra of its diagonal blocks. It therefore has the eigenvalues $e^{\pm i\theta}$ but no eigenvalue equal to $1$.

If all Jordan blocks associated with the eigenvalue $1$ are
one-dimensional, replace two copies of the scalar block $1$ by the
rotation block
$$R_\theta=
\begin{pmatrix}
\cos\theta&-\sin\theta\\
\sin\theta&\cos\theta
\end{pmatrix},
$$
and replace every remaining copy of $1$ by $1+\delta_s$, with
$\delta_s\neq0$ arbitrarily small.

In either case, choosing $\theta\neq0$, $\delta$, and the $\delta_s$
sufficiently small produces a real matrix arbitrarily close to the
original Jordan form, with no eigenvalue equal to $1$ and with
$e^{\pm i\theta}$ in its spectrum. Since $GL(n,\mathbb R)$ is open and
$A\in GL(n,\mathbb R)$, all perturbation parameters may also be chosen
so that the resulting matrix remains invertible.

Conjugating back to the original coordinates gives the required real matrix $A'$.
\end{proof}

\begin{proof}[\textbf{\emph{Proof of Theorem \ref{Aff Par Str}}}]
    Fix $(A, v)\in \mathcal{P}_n$. Let $\mu_1, \dots, \mu_n$ be the complex eigenvalues of $A$, counted with algebraic multiplicity. For every $r>0$, the eigenvalues of $rA$ are $r\mu_1, \dots, r\mu_n$. Thus $rA$ fails to be hyperbolic only when $r|\mu_i|=1$ for at least one $i$. Since $A$ is invertible and has only finitely many eigenvalues, there are only finitely many forbidden positive values of $r$. We can therefore choose $r_m\to 1$ such that every $r_m A$ is hyperbolic. By the classification theorem, $f_{r_mA, v}$ admits a complete Anosov metric. Hence $(r_mA, v)\in \mathcal{A}_n$ and $(r_m A, v)\to (A, v)$. Therefore $\overline{\mathcal{A}_n}=\mathcal{P}_n$.

    Hyperbolicity is an open condition in $GL(n, \mathbb{R})$. Hence if $A_0$ is hyperbolic, every matrix $A$ sufficiently close to $A_0$ is also hyperbolic. Therefore every pair $(A, v)$ sufficiently close to $(A_0, v_0)$ belongs to $\mathcal{A}_n$ independently of the perturbation of $v$. Thus
    $$\mathcal{H}_n\subset \operatorname{Int}_{\mathcal{P}_n}(\mathcal{A}_n).$$

    Let $(A_0, v_0)\in \mathcal{D}^{reg}_n$. Since 1 is an algebraically simple eigenvalue of $A_0$, standard perturbation theory gives a neighborhood $\mathcal{V}$ of $A_0$ and a continuous simple eigenvalue $A\mapsto \lambda(A)$ such that $\lambda(A_0)=1$ and $\lambda(A)$ is the unique eigenvalue of $A$ near 1. The neighborhood may be chosen so that all other eigenvalues of $A$ remain uniformly separated from the unit circle. Since $A$ is real and the eigenvalue near 1 is unique, $\lambda(A)$ is real. We distinguish two cases: If $\lambda(A)\neq 1$, then $|\lambda(A)|\neq 1$, since $\lambda(A)$ is real and sufficiently close to 1. The remaining eigenvalues are also off the unit circle. Hence $A$ is hyperbolic, so $(A, v)\in \mathcal{A}_n$. Suppose now that $\lambda(A)=1$. Since this eigenvalue is simple, the left eigenspace $\operatorname{ker}(I-A^T)$ is one-dimensional. After shrinking $\mathcal V$ to a ball, choose a continuous family of unit left eigenvectors $\ell_A$ satisfying $\ell_A \circ A=\lambda(A)\ell_A$. In the case $\lambda(A)=1$, this gives $\ell_A((I-A)x)=0.$ Hence $\operatorname{Im}(I-A)\subset\ker\ell_A.$ Since the eigenvalue $1$ is algebraically simple, $\operatorname{rank}(I-A)=n-1.$ Thus $\operatorname{Im}(I-A)$ has codimension one. Since $\ell_A$ is a nonzero linear functional, $\ker\ell_A$ also has codimension one. Therefore the inclusion above is an equality $\operatorname{Im}(I-A)=\ker\ell_A.$

    At $(A_0, v_0)$, the regular drift condition gives $v_0\notin \operatorname{Im}(I-A_0)$, or $\ell_{A_0}(v_0)\neq 0$. By continuity, $\ell_A(v)\neq 0$ for every $(A, v)$ sufficiently close to $(A_0, v_0)$. Hence $v \notin \operatorname{Im}(I-A)$. By Theorem \ref{thm:mainresult}, $f_{A, v}$ admits a complete Anosov metric. Thus every regular drift pair is an interior point
    $$\mathcal{D}^{reg}_n\subset \operatorname{Int}_{\mathcal{P}_n}(\mathcal{A}_n).$$

    Let $(A_0, v_0)\in \mathcal{A}_n$, assume that $A_0$ is not hyperbolic and suppose that $(A_0, v_0)\notin \mathcal{D}^{reg}_n$. Since $A_0$ is not hyperbolic, the classification theorem implies $v_0\notin \operatorname{Im}(I-A_0)$. It follows that $1\in \operatorname{Spec}(A_0)$, because otherwise $I-A_0$ would be invertible and $\operatorname{Im}(I-A_0)=\mathbb{R}^n$. Since $(A_0, v_0)$ is not a regular drift pair, either $A_0$ has another eigenvalue on the unit circle or the eigenvalue 1 is not algebraically simple. By Lemma \ref{lem:specpert}, there exists a sequence $A_m\to A_0$ such that $1\notin \operatorname{Spec}(A_m)$, but $\operatorname{Spec}(A_m)\cap \mathbb{S}^1\neq \emptyset$. Since $1\notin \operatorname{Spec}(A_m)$, the operator $I-A_m$ is invertible. Hence $f_{A_m, v_0}$ has a unique fixed point. At this fixed point, the derivative is $A_m$. Since $A_m$ has an eigenvalue of modulus one, the fixed point is not hyperbolic. By the fixed-point obstruction, $f_{A_m, v_0}$ cannot be Anosov with respect to any complete Riemannian metric. Thus $(A_m, v_0)\notin \mathcal{A}_n$, while $(A_m, v_0)\to (A_0, v_0)$. Therefore $(A_0, v_0)$ is not an interior point. Hence we have
    $$\operatorname{Int}_{\mathcal{P}_n}(\mathcal{A}_n)=\mathcal{H}_n \cup \mathcal{D}^{reg}_n.$$
    Since $\mathcal{A}_n$ is dense in $\mathcal{P}_n$, $\overline{\mathcal A_n}=\mathcal P_n.$ Therefore
    $$\partial_{\mathcal{P}_n}\mathcal{A}_n=\overline{\mathcal{A}_n}\setminus \operatorname{Int}_{\mathcal{P}_n}(\mathcal{A}_n)=\mathcal{P}_n\setminus (\mathcal{H}_n\cup \mathcal{D}^{reg}_n).$$
\end{proof}

\begin{corollary}[Index spectrum near a regular drift parameter]
\label{cor:index-wall-crossing}
Let $(A_0,v_0)\in\mathcal D_n^{\mathrm{reg}}$. Let $s$ be the total algebraic multiplicity of the eigenvalues of $A_0$, different from $1$, which lie inside the unit disk.

Then there exist a neighborhood $\mathcal U$ of $(A_0,v_0)$ contained in $\mathcal A_n$ and a smooth real eigenvalue branch
$A\mapsto\lambda(A)$, with $\lambda(A_0)=1$, such that
$$ \operatorname{Ind}(f_{A,v}) =
\begin{cases}
\{s+1\}, & \lambda(A)<1,\\[1mm]
\{1,\ldots,n-1\}, & \lambda(A)=1,\\[1mm]
\{s\}, & \lambda(A)>1
\end{cases}
$$
for every $(A,v)\in\mathcal U$.
\end{corollary}

\begin{proof}
Since $(A_0,v_0)$ is a regular drift pair, Theorem~\ref{Aff Par Str} shows that it is an interior point of $\mathcal A_n$. Hence there exists a neighborhood $\mathcal U_0$ of $(A_0,v_0)$ contained in $\mathcal A_n$.

The eigenvalue $1$ of $A_0$ is algebraically simple and all the other eigenvalues are separated from the unit circle. Standard perturbation theory therefore gives, after shrinking a neighborhood of $A_0$, a smooth real eigenvalue branch $\lambda(A)$ with $\lambda(A_0)=1$. The remaining eigenvalues stay uniformly separated from the unit circle and remain on the same side of it. Shrinking $\mathcal U_0$ if necessary, we obtain a neighborhood $\mathcal U\subset\mathcal A_n$
with these properties.

If $\lambda(A)<1$, then $A$ is hyperbolic and its stable generalized eigenspace has dimension $s+1$. By Theorem~\ref{thm:index},
$$
\operatorname{Ind}(f_{A,v})=\{s+1\}.
$$
If $\lambda(A)>1$, the distinguished eigenvalue belongs to the
unstable spectrum, so Theorem~\ref{thm:index} gives
$$
\operatorname{Ind}(f_{A,v})=\{s\}.
$$

Finally, suppose that $\lambda(A)=1$. Then $A$ is not hyperbolic.
Since $(A,v)\in\mathcal U\subset\mathcal A_n$, Theorem~\ref{thm:mainresult} implies
$$
v\notin\operatorname{Im}(I-A).
$$
The second part of Theorem~\ref{thm:index} therefore yields
$$
\operatorname{Ind}(f_{A,v}) = \{1,\ldots,n-1\}.
$$
\end{proof}

\begin{example}[An explicit family]
Let $B\in GL(n-1,\mathbb R)$ be hyperbolic and let
$s=\dim E_B^s$. For $\tau\ne 0$ and $w\in\mathbb R^{n-1}$, consider
$$ f_\varepsilon(t,z) = ((1+\varepsilon)t+\tau,Bz+w).
$$
At $\varepsilon=0$, the pair is a regular drift parameter, and the distinguished eigenvalue branch is
$$ \lambda(\varepsilon)=1+\varepsilon. $$
Hence Corollary~\ref{cor:index-wall-crossing} gives, for all
sufficiently small $|\varepsilon|$,
$$ \operatorname{Ind}(f_\varepsilon)=
\begin{cases}
\{s+1\}, & \varepsilon<0,\\[1mm]
\{1,\ldots,n-1\}, & \varepsilon=0,\\[1mm]
\{s\}, & \varepsilon>0.
\end{cases}
$$
\end{example}

\section{Stability at Infinity}\label{Stability at Infinity}
On a noncompact manifold, $C^1$ stability depends on the topology used to control perturbations at infinity. We recall the two topologies used below.

\begin{definition}[Two-sided weak and strong Whitney $C^1$ topologies]
We endow $\operatorname{Diff}^1(\mathbb R^n)$ with the two-sided weak
(resp. strong Whitney) $C^1$ topology, namely the initial topology induced by $
h\longmapsto (h,h^{-1}) $
from the product of the corresponding $C^1$ topologies on
$$
C^1(\mathbb R^n,\mathbb R^n)\times
C^1(\mathbb R^n,\mathbb R^n).
$$ A sequence $f_m\in\operatorname{Diff}^1(\mathbb R^n)$ converges to
$f\in\operatorname{Diff}^1(\mathbb R^n)$ in the two-sided weak
$C^1_{\mathrm{loc}}$ topology if, for every compact set
$K\subset\mathbb R^n$,
$$
\begin{aligned}
\sup_{x\in K}\bigl(
&|f_m(x)-f(x)|
+\|Df_m(x)-Df(x)\|\\
&+|f_m^{-1}(x)-f^{-1}(x)|
+\|D(f_m^{-1})_x-D(f^{-1})_x\|
\bigr)
\longrightarrow 0.
\end{aligned}
$$

For the two-sided strong Whitney $C^1$ topology, let $\varepsilon_+,\varepsilon_-:\mathbb R^n\to(0,\infty)$ be continuous. A basic neighborhood of $f\in\operatorname{Diff}^1(\mathbb R^n)$ consists of all $h\in\operatorname{Diff}^1(\mathbb R^n)$ such that, for every $x\in\mathbb R^n$,
$$
|h(x)-f(x)|+\|Dh_x-Df_x\|<\varepsilon_+(x)
$$
and
$$
|h^{-1}(x)-f^{-1}(x)|+ \|D(h^{-1})_x-D(f^{-1})_x\|<\varepsilon_-(x).
$$
\end{definition}

\begin{proof}[\textbf{\emph{Proof of Theorem \ref{thm:weak}}}]
By Proposition~\ref{pro:smoconju}, there exists a smooth diffeomorphism
$H:\mathbb R^n\to\mathbb R^n$ such that
$$ H\circ f_{A,v}\circ H^{-1}=T_\sigma, $$
where $ T_\sigma(t,z)=(t+1,R_\sigma z)$
and $$R_\sigma=
\begin{cases}
I_{n-1}, & \det A>0,\\
\operatorname{diag}(-1,1,\ldots,1), & \det A<0.
\end{cases}
$$
For $\varepsilon\ne0$ sufficiently small, define
$$ T_{\sigma,\varepsilon}(t,z)=((1+\varepsilon)t+1,R_\sigma z).
$$
This map has the fixed point
$$ p_\varepsilon =\left(-\frac1\varepsilon,0\right),$$
and
$$ DT_{\sigma,\varepsilon}(p_\varepsilon)=(1+\varepsilon)\oplus R_\sigma.
$$
Since $R_\sigma$ has an eigenvalue of modulus one, this fixed point is not hyperbolic. Hence $T_{\sigma,\varepsilon}$ is not
Anosov-realizable.

We have
$$ T_{\sigma,\varepsilon}(t,z)-T_\sigma(t,z)=(\varepsilon t,0),
$$
and
$$ DT_{\sigma,\varepsilon}-DT_\sigma=
\begin{pmatrix}
\varepsilon&0\\
0&0
\end{pmatrix}.
$$
Moreover,
$$ T_{\sigma,\varepsilon}^{-1}(t,z)
=\left(\frac{t-1}{1+\varepsilon},R_\sigma^{-1}z\right),
$$
so
$$T_{\sigma,\varepsilon}^{-1}(t,z)-T_\sigma^{-1}(t,z)=\left(
-\frac{\varepsilon}{1+\varepsilon}(t-1),0\right).
$$
The corresponding derivative difference also tends to zero as
$\varepsilon\to0$, and all derivatives of order at least two vanish. Therefore
$$
T_{\sigma,\varepsilon}\longrightarrow T_\sigma
$$
together with their inverses in $C^\infty$ on compact subsets.

Now define
$$h_\varepsilon=H^{-1}\circ T_{\sigma,\varepsilon}\circ H.
$$
If $h_\varepsilon$ were Anosov-realizable, Lemma~\ref{lem:smopre} applied to the smooth conjugacy $H$ would imply that $T_{\sigma,\varepsilon}$ is Anosov-realizable, a contradiction. Thus $h_\varepsilon\notin\mathscr D_n$.

Since $H$ and $H^{-1}$ are fixed smooth diffeomorphisms and send
compact sets to compact sets, the preceding convergence gives
$$ h_\varepsilon\longrightarrow f_{A,v},\qquad
h_\varepsilon^{-1}\longrightarrow f_{A,v}^{-1}
$$
in $C^\infty$ on compact subsets. By Theorem~\ref{thm:mainresult},
$f_{A,v}\in\mathscr D_n$. Hence $f_{A,v}$ is a boundary point of
$\mathscr D_n$ in the two-sided weak $C^1_{\mathrm{loc}}$ topology.

\end{proof}
\begin{lemma}[Smooth relative approximation of continuous metrics]
\label{lem:smooth-relative-metric}
Let $M$ be a smooth manifold and let $g_*$ be a continuous Riemannian
metric on $M$. For every continuous function
$\delta:M\to(0,1)$, there exists a smooth Riemannian metric
$\hat g$ such that
$$(1-\delta(x))g_{*,x}\le\hat g_x \le
(1+\delta(x))g_{*,x}$$
as quadratic forms for every $x\in M$.
\end{lemma}
\begin{proof}
This follows from the smooth approximation theorem for continuous
sections; see \cite{Wockel}. Apply it to the bundle of positive
definite symmetric bilinear forms and to the open neighborhood of the graph of $g_*$ defined by
$$ (1-\delta(x))g_{*,x} < b < (1+\delta(x))g_{*,x}.$$
A smooth section whose graph lies in this neighborhood is the
required metric $\hat g$. The strict inequalities above imply the
non-strict inequalities stated in the lemma.
\end{proof}

\begin{lemma}[Complete adapted metric]\label{lem:commetric}
    Let $f\colon \mathbb{R}^n\to \mathbb{R}^n$ be Anosov with respect to a complete Riemannian metric $g$ with splitting $T\mathbb R^n=E^s \oplus E^u$. Then there exist a continuous Riemannian norm $\|\cdot\|_*$, a complete smooth Riemannian metric $\hat g$, a constant $q\in(0,1)$, and constants $0<c\le C<\infty$ such that:
    \begin{enumerate}
    \item[\emph{\text{1.}}] $E^s$ and $E^u$ are orthogonal with respect to
    $\|\cdot\|_*$;
    \item[\emph{\text{2.}}] $\|Df_xv^s\|_{*,f(x)} \le q\|v^s\|_{*,x}, \qquad v^s\in E^s(x);$
    \item[\emph{\text{3.}}] $\|D(f^{-1})_xv^u\|_{*,f^{-1}(x)} \le q\|v^u\|_{*,x}, \qquad v^u\in E^u(x);$
    \item[\emph{\text{4.}}] $c\|v\|_{*,x}\le \|v\|_{\hat g,x}\le C\|v\|_{*,x}$ for every $x\in\mathbb R^n$ and $v\in T_x\mathbb R^n$.
    \end{enumerate}
    Moreover, the angle between $E^s(x)$ and $E^u(x)$ measured with respect to $\hat{g}$ is bounded away from zero.
\end{lemma}
\begin{proof}
    Let $C_0\geq1$ and $\lambda_0\in(0,1)$ be Anosov constants for $g$. Choose a number $\alpha$ such that $\lambda_0<\alpha<1$. For $v^s \in E^s(x)$, define
    $$ \| v^s \|_{*,x}^2 = \sum_{m=0}^{\infty} \alpha^{-2m} \| D(f^m)_x v^s \|_{g, f^m(x)}^2.$$
    For $v^u \in E^u(x)$, define
    $$\| v^u \|_{*,x}^2 = \sum_{m=0}^{\infty} \alpha^{-2m} \| D(f^{-m})_x v^u \|_{g, f^{-m}(x)}^2.$$
    Declare $E^s(x)$ and $E^u(x)$ to be orthogonal and define
    $$\| v^s + v^u \|_{*,x}^2 = \| v^s \|_{*,x}^2 + \| v^u \|_{*,x}^2.$$
   The Anosov estimates give 
   $$\alpha^{-2m}\|D(f^m)_xv^s\|_{g,f^m(x)}^2 \le C_0^2 \left(\frac{\lambda_0}{\alpha}\right)^{2m} \|v^s\|_{g,x}^2. $$
   and the analogous estimate on $E^u$. Since $\lambda_0/\alpha<1$, both series converge uniformly on the $g$-unit bundles of $E^s$ and $E^u$. Therefore $\|\cdot\|_*$ is a continuous Riemannian norm. For $v^s \in E^s(x)$, we have
   $$\| Df_x v^s \|_{*, f(x)}^2 = \sum_{m=0}^{\infty} \alpha^{-2m} \| D(f^{m+1})_x v^s \|_g^2
= \alpha^2 \sum_{m=1}^{\infty} \alpha^{-2m} \| D(f^m)_x v^s \|_g^2
\leq \alpha^2 \| v^s \|_{*, x}^2.
$$
Hence $ \| Df_x v^s \|_{*, f(x)} \leq \alpha \| v^s \|_{*, x}.$
Similarly, $\|D(f^{-1})_x v^u\|_{*,f^{-1}(x)}\le\alpha\|v^u\|_{*,x}.$ Set $q:=\alpha$. Then $q\in (0, 1)$ and the preceding two inequalities give items (2) and (3). The zeroth terms in the two defining series give
$$
\| v^s \|_{*, x} \geq \| v^s \|_{g, x}, \quad \| v^u \|_{*, x} \geq \| v^u \|_{g, x}.
$$
Therefore, for $v = v^s + v^u$, $$
\| v \|_{g, x} \leq \| v^s \|_{g, x} + \| v^u \|_{g, x}
\leq \sqrt{2} \| v \|_{*, x}.$$
Thus, $\| v \|_{*, x} \geq \frac{1}{\sqrt{2}} \| v \|_{g, x}.$ 

The norm $\|\cdot\|_*$ is continuous and does not need to be smooth. Choose a constant $\delta\in(0,1)$. By Lemma~\ref{lem:smooth-relative-metric}, there exists a smooth Riemannian metric $\hat g$ satisfying
\begin{equation}\label{eq:normg}
(1-\delta) \| v \|_*^2 \leq \| v \|_{\hat{g}}^2 \leq (1+\delta) \| v \|_*^2. 
\end{equation}
Thus we may take
$$c=\sqrt{1-\delta}, \qquad C=\sqrt{1+\delta}. $$
Moreover, since $\|v\|_{*,x}\ge\frac{1}{\sqrt2}\|v\|_{g,x}$, we have
$\hat g \geq \frac{1-\delta}{2}\,g$. So $\hat{g}$ is complete. Finally, $E^s$ and $E^u$ are orthogonal in the $*$-metric, and $\hat{g}$ is uniformly equivalent to that metric. Therefore their $\hat{g}$-angle is uniformly bounded away from zero.

\end{proof}

\begin{proof}[\textbf{\emph{Proof of Theorem \ref{thm:whit}}}]
By Lemma~\ref{lem:commetric}, there exist a continuous
adapted norm $\|\cdot\|_*$, a complete smooth Riemannian metric
$\hat g$, and a constant $q\in(0,1)$ and constants $0<c\le C$ satisfying properties (1)-(4) of that lemma.

We record the uniform graph-transform estimates for $f$. For every $x\in \mathbb R^n$, write 
$$S_x:= Df_x|_{E^s(x)}, \qquad U_x:=Df_x|_{E^u(x)}.$$
By the one-step estimates in Lemma~\ref{lem:commetric}, we have $\left\|S_x\right\|_*\leq q$ and $\left\|U^{-1}_x\right\|_*\leq q$. 
The latter is equivalent to $$
\|U_xu\|_{*,f(x)}
\ge
q^{-1}\|u\|_{*,x},
\qquad
u\in E^u(x).$$
Choose $a\in (0, 1)$ sufficiently small that 
$$\mu_*:=\frac{1}{q\sqrt{1+a^2}}>1.$$
At each $x$, consider the closed ball
$$\mathcal{B}^u_x(a)=\{L:E^u(x)\longrightarrow E^s(x):\left\|L\right\|_*\leq a\}.$$
The graph of $L\in \mathcal{B}^u_x(a)$ is
$$\operatorname{graph}L=\{Lu+u: u\in E^u(x)\}.$$
Since $Df_x$ preserves $E^s$ and $E^u$, it sends this graph to the graph of $\Gamma_{f,x}^{u}(L) = S_x L U_x^{-1}$
Therefore, for $L_1, L_2\in \mathcal{B}^u_x(a)$,
\begin{equation}\label{eq:boundl}
\left\| \Gamma_{f,x}^{u}(L_1) - \Gamma_{f,x}^{u}(L_2)\right\|_* \leq q^2 \left\| L_1 - L_2 \right\|_*.
\end{equation}
Since $\Gamma_{f,x}^u(0)=0$, it follows that
$$\|\Gamma_{f,x}^u(L)\|_*\le q^2\|L\|_* \le q^2a<a,$$
for every $L\in\mathcal B_x^u(a)$. Thus the graph transform maps
$\mathcal B_x^u(a)$ strictly into itself. Moreover, if $v=Lu+u\in \operatorname{graph}L$, then orthogonality of $E^s(x)$ and $E^u(x)$ gives 
$$\|v\|_{*,x}^2 = \|Lu\|_{*,x}^2+\|u\|_{*,x}^2 \le (1+a^2)\|u\|_{*,x}^2. $$
Moreover, $Df_xv=S_xLu+U_xu,$ and these two summands belong to the orthogonal spaces $E^s(f(x))$ and $E^u(f(x))$, respectively. Therefore
$$\begin{aligned}
\|Df_xv\|_{*,f(x)}
&\ge \|U_xu\|_{*,f(x)}\\
&\ge q^{-1}\|u\|_{*,x}\\
&\ge \frac{1}{q\sqrt{1+a^2}}\|v\|_{*,x}\\
&= \mu_*\|v\|_{*,x}.
\end{aligned} $$
The same argument applied to $Df^{-1}$ gives a stable graph transform on maps
$$L:E^s(x)\longrightarrow E^u(x), \qquad \left\|L\right\|_* \leq a,$$
whose graph transform contracts by $q^2$, while vectors in the
corresponding graphs are expanded by at least $\mu_*$ under
$Df^{-1}$.

Choose constants $\vartheta$ and $\mu$ satisfying $q^2<\vartheta<1$ and $1<\mu <\mu_*$.  
Fix $x\in\mathbb R^n$. Let $y$ be sufficiently close to $f(x)$, and 

let $\mathcal L:T_x\mathbb R^n\longrightarrow T_y\mathbb R^n$
be a linear isomorphism sufficiently close to $Df_x$. With respect to
the decompositions
$$
T_x\mathbb R^n=E^s(x)\oplus E^u(x),
\qquad
T_y\mathbb R^n=E^s(y)\oplus E^u(y),
$$
write $$
\mathcal L=
\begin{pmatrix}
\mathcal L_{ss}&\mathcal L_{su}\\
\mathcal L_{us}&\mathcal L_{uu}
\end{pmatrix}.$$

For $L\in\mathcal B_x^u(a)$ and $u\in E^u(x)$, we have
$$ \mathcal L(Lu+u)=(\mathcal L_{ss}L+\mathcal L_{su})u+
(\mathcal L_{us}L+\mathcal L_{uu})u.$$
Provided that $\mathcal L_{us}L+\mathcal L_{uu}$
is invertible, the image of $\operatorname{graph}L$ is the graph of
\begin{equation}\label{eq:lfor}
\Gamma_{\mathcal L}^u(L)=
(\mathcal L_{ss}L+\mathcal L_{su})
(\mathcal L_{us}L+\mathcal L_{uu})^{-1}.
\end{equation}
At $y=f(x)$ and $\mathcal L=Df_x$, the off-diagonal blocks vanish,
and hence $\Gamma_{\mathcal L}^u(L)= S_xLU_x^{-1}.$
The expression in \eqref{eq:lfor}, as well as its derivative with
respect to $L$, depends continuously on the point $y$, the linear map $\mathcal L$, and the graph map $L$ whenever the second factor is invertible. For $\mathcal L=Df_x$, we have $$\|\Gamma_{\mathcal L}^u(L)\|_* \le q^2 a<a, \qquad \|D_L\Gamma_{\mathcal L}^u\| \le q^2<\vartheta, $$
and $$\|\mathcal L v\|_{*,f(x)}\ge \mu_*\|v\|_{*,x}>\mu\|v\|_{*,x}$$
for every $L\in\mathcal B_x^u(a)$ and every $v\in\operatorname{graph}L$.

The expansion estimate also depends on the vector in the graph. For this purpose, consider
$$\mathcal K_x^u(a):=\left\{(L,v):L\in\mathcal B_x^u(a),\
v\in\operatorname{graph}L,\
\|v\|_{*,x}=1 \right\}.$$
This set is compact. Since the graph-transform inequalities are strict on $\mathcal B_x^u(a)$ and the expansion inequality is strict on $\mathcal K_x^u(a)$, there exists $\delta_+(x)>0$ such that, whenever
$$
|y-f(x)|+\|\mathcal L-Df_x\|<\delta_+(x),
$$
the operator $\mathcal L_{us}L+\mathcal L_{uu}$ is invertible for every $L\in\mathcal B_x^u(a)$,
$$
\Gamma_{\mathcal L}^u(L)\in\mathcal B_y^u(a),\qquad\|D_L\Gamma_{\mathcal L}^u\|\le\vartheta,
$$
and
$$ \|\mathcal L v\|_{*,y}\ge\mu\|v\|_{*,x}$$
for every $L\in\mathcal B_x^u(a)$ and every $v\in\operatorname{graph}L$.

Since the graph ball is convex, the mean-value inequality gives
$$\|\Gamma_{\mathcal L}^u(L_1)-\Gamma_{\mathcal L}^u(L_2)\|_*\le
\vartheta\|L_1-L_2\|_*$$
for all $L_1,L_2\in\mathcal B_x^u(a)$.

Applying the same argument to $f^{-1}$, for every
$x\in\mathbb R^n$ there exists a positive inverse tolerance
$\delta_-(x)$ such that the corresponding stable graph-transform
estimates hold.

For each $x\in\mathbb R^n$, let $r_+(x)$ be the supremum of the
numbers $0<r\leq 1$ for which the preceding forward estimates hold whenever
$$
|y-f(x)|+\|\mathcal L-Df_x\|<r.
$$
Define $r_-(x)$ analogously for $f^{-1}$. Since $\delta_+(x)>0$ and $\delta_-(x)>0$, both functions are positive. The strict inequalities above and the continuous dependence on the
base point imply that $r_+$ and $r_-$ are lower semicontinuous.

For $\sigma\in\{+,-\}$, define
$$
\varepsilon_\sigma(x)=
\frac12\inf_{y\in\mathbb R^n}
\bigl(r_\sigma(y)+|x-y|\bigr).
$$
The function $\varepsilon_\sigma$ is continuous, since
$$
|\varepsilon_\sigma(x)-\varepsilon_\sigma(x')|
\le
\frac12|x-x'|.
$$
It is also positive. Indeed, the positivity and lower
semicontinuity of $r_\sigma$ give a positive lower bound for
$r_\sigma$ near $x$, while $|x-y|$ is bounded away from zero outside a sufficiently small neighborhood of $x$. Finally, taking $y=x$ in the infimum gives
$$
0<\varepsilon_\sigma(x) \le \frac12 r_\sigma(x) < r_\sigma(x). $$

Thus all the preceding estimates hold whenever
$$ |h(x)-f(x)|+\|Dh_x-Df_x\|<\varepsilon_+(x) $$
and
$$ |h^{-1}(x)-f^{-1}(x)| +\|D(h^{-1})_x-D(f^{-1})_x\| <\varepsilon_-(x).$$

Let $\mathcal U$ be the resulting two-sided strong Whitney $C^1$
neighborhood of $f$. Fix $h\in \mathcal{U}$, we construct its unstable bundle. Let $\mathcal{X}^u_a$ be the space of the continuous bundle maps $L:E^u \longrightarrow E^s$ such that $\left\|L\right\|_{\infty}:=\sup_{x\in \mathbb{R}^n}\left\|L_x\right\|_*\leq a.$

Equipped with the uniform distance, $\mathcal{X}^u_a$ is a complete metric space. Indeed, a uniformly Cauchy sequence of continuous bundle maps has a uniformly convergent limit, and the limit is continuous. For $L\in \mathcal{X}^u_a$, the derivative $Dh_x$ maps graph $L_x$ to a graph over $E^u(h(x))$. Define $\Gamma^u_h L$ by
$$Dh_x(\operatorname{graph}L_x)=\operatorname{graph}((\Gamma^u_h L)_{h(x)}).$$
Since $h$ is bijective, this defines a bundle map over every point of $\mathbb R^n$. The choice of $\mathcal{U}$ implies that $\Gamma^u_h: \mathcal{X}^u_a\longrightarrow \mathcal{X}^u_a$ and
$$\left\|\Gamma^u_h L_1-\Gamma^u_h L_2\right\|_{\infty}\leq \vartheta \left\|L_1-L_2\right\|_{\infty}.$$
Since $\vartheta<1$, Banach's fixed point theorem gives a unique fixed point $L^u_h\in \mathcal{X}^u_a$. Define $E^u_h(x):=\operatorname{graph} L^u_h(x)$. The fixed-point equation implies $Dh_x E^u_h(x)=E^u_h(h(x))$. Thus $E^u_h$ is a continuous $Dh$-invariant subbundle and $\dim E^u_h=\dim E^u=n-k$. 

Applying the same construction to $h^{-1}$ gives a continuous $Dh$-invariant stable bundle $E^s_h(x)=\operatorname{graph}L^s_h (x)$, where $L^s_h: E^s\to E^u$ and $\left\|L^s_h\right\|_{\infty}\leq a$. Hence $\dim E^s_h=\dim E^s=k$. Because $a<1$, the two graph bundles are transverse. Suppose that a vector belongs to both. Then it can be written both as $u+L^u_h u$ and as $s+L^s_h s$. Comparing the $E^s$ and $E^u$ components gives
$$\left\|s\right\|_* \le a \left\|u\right\|_*, \qquad \left\|u\right\|_* \le a \left\|s\right\|_*. $$
Thus $\left\|u\right\|_*\leq a^2 \left\|u\right\|_*$. Since $a<1$, this implies $u=0$ and $s=0$. Thus $E^s_h(x)\cap E^u_h(x)=\{0\}$. Their dimensions add to $n$, so $T_x\mathbb R^n =E^s_h(x)\oplus E^u_h(x)$ for every $x$. 

We now prove the hyperbolic estimates. By the choice of $\mathcal{U}$, every vector $v^u\in E^u_h(x)$ satisfies $\| Dh_x v^u \|_{*, h(x)} \ge \mu \| v^u \|_{*, x}.$
Then we iterate and obtain $\| D(h^m)_x v^u \|_{*, h^m(x)} \ge \mu^m \| v^u \|_{*, x}.$
That is, for $w^u \in E_h^u(h^m(x))$, we have
\begin{equation}\label{eq:newmu}
    \| D(h^{-m})_{h^m(x)} w^u \|_{*, x} \le \mu^{-m} \| w^u \|_{*, h^m(x)}
\end{equation}
Similarly, the inverse graph-transform estimates imply
\begin{equation}\label{eq:newv}
    \| D(h^m)_x v^s \|_{*, h^m(x)} \le \mu^{-m} \| v^s \|_{*, x}
\end{equation}
for $v^s \in E_h^s(x)$. Finally, by \eqref{eq:normg}, \eqref{eq:newmu} and \eqref{eq:newv},
$$\| D(h^m)_x v^s \|_{\hat{g}, h^m(x)} \le \frac{C}{c} \mu^{-m} \| v^s \|_{\hat{g}, x},$$
and
$$\| D(h^{-m})_x v^u \|_{\hat{g}, h^{-m}(x)} \le \frac{C}{c} \mu^{-m} \| v^u \|_{\hat{g}, x}. $$
Thus $h$ is Anosov with respect to the same complete metric $\hat{g}$ with Anosov constants 
$$C_*=\frac{C}{c}, \lambda_*=\mu^{-1}<1,$$
which are independent of $h\in \mathcal U$. The dimensions of its stable and unstable bundles are $k$ and $n-k$, respectively

\end{proof}
\begin{corollary}[Compactly supported nonlinear Anosov perturbations] Let $f=f_{A, v}$ be an affine diffeomorphism of $\mathbb R^n$ which is Anosov with respect to a complete Riemannian metric. Then every neighborhood of $f$ in the two-sided strong Whitney $C^1$ topology contains a non-affine diffeomorphism $h$ with the following properties:
\begin{enumerate}
    \item[\emph{\text{1.}}]  $h=f$ outside a compact subset of $\mathbb R^n$;
    \item[\emph{\text{2.}}] $h$ is Anosov with respect to a complete Riemannian metric;
    \item[\emph{\text{3.}}] $h$ has the same stable index as $f$.
\end{enumerate}
More precisely, $h$ may be chosen to be Anosov with respect to the complete adapted metric $\hat{g}$.
\end{corollary}
\begin{proof}
    Let $\mathcal V$ be an arbitrary two-sided strong Whitney $C^1$ neighborhood of $f$, and let $\mathcal U$ be the neighborhood furnished by Theorem~\ref{thm:whit}. Since $\mathcal V\cap\mathcal U$ is a neighborhood of $f$, there exist positive continuous functions
    $$\varepsilon_+,\varepsilon_-:\mathbb R^n\longrightarrow(0,\infty)$$
    such that the basic neighborhood $\mathcal W$ determined by $\varepsilon_+$ and $\varepsilon_-$ satisfies $\mathcal W\subset\mathcal V\cap\mathcal U.$

    Choose a nonzero smooth compactly supported vector field $X$ on $\mathbb{R}^n$. Since $X$ is compactly supported, it is complete. Let $\Phi_\tau:\mathbb R^n \to \mathbb R^n$ denote its time-$\tau$ flow. Let $K:=\operatorname{supp} X$. For every $\tau$, $\Phi_\tau$ is a smooth diffeomorphism and $\Phi_\tau=\operatorname{id}$ on $\mathbb R^n \setminus K$. Moreover, $\Phi_\tau\to \operatorname{id}$ in $C^1$ on compact sets as $\tau\to 0$ and $\Phi_\tau^{-1}=\Phi_{-\tau}$. Define $h_{\tau}:=\Phi_\tau\circ f$. If $x\notin f^{-1}(K)$, then $f(x)\notin K$, so $\Phi_\tau(f(x))=f(x)$. Hence $h_\tau(x)=f(x)$ outside the compact set $f^{-1}(K)$. The set $f^{-1}(K)$ is compact because $f$ is a homeomorphism. The inverse map is $h^{-1}_\tau=f^{-1}\circ \Phi_{-\tau}$. Thus $h^{-1}_\tau=f^{-1}$ outside $K$. We claim that $h_{\tau}\to f$ in the two-sided strong Whitney $C^1$ topology as $\tau\to 0$.

    The forward maps and derivatives agree exactly outside the fixed compact set $f^{-1}(K)$. On $f^{-1}(K)$, the convergence $\Phi_\tau \circ f \to f$ holds in the $C^1$ topology. Similarly, the inverse map and derivatives agree exactly outside $K$, while on $K$, $f^{-1}\circ \Phi_{-\tau}\to f^{-1}$ in $C^1$.

    Since $f^{-1}(K)$ and $K$ are compact,
    $$\min_{f^{-1}(K)} \varepsilon_+>0, \qquad \min_K \varepsilon_->0.$$
    Uniform $C^1$ convergence on these two compact sets implies that $h_\tau\in \mathcal{W}$ for all sufficiently small $|\tau|$. 
    
    Choose $x_0\in \mathbb R^n$ such that $X(x_0)\neq 0$. The orbit curve $\tau \longmapsto \Phi_\tau(x_0)$ satisfies
    $$\left.\frac{d}{d \tau}\right|_{\tau=0} \Phi_\tau(x_0) = X(x_0) \neq 0.$$
    Choose a linear functional $\ell:\mathbb R^n\longrightarrow \mathbb R$ such that $\ell(X(x_0))\neq 0$. Then the real-valued function $\tau\longmapsto \ell(\Phi_\tau(x_0))$ has nonzero derivative at $\tau=0$ and is therefore strictly monotone on a sufficiently small interval about $0$. Hence $\Phi_\tau(x_0) \neq x_0 $ for every sufficiently small $\tau\neq 0$ and thus $\Phi_\tau\neq \operatorname{id}$. Fix such a $\tau \neq 0$ with $h_{\tau}\in \mathcal{W}$. We now show that $h_{\tau}$ is not affine. 
    
    Suppose that $h_\tau$ is affine. Since both $h_\tau$ and $f$ are affine, their difference is affine.. But $h_{\tau} -f$ vanishes on the nonempty open set $\mathbb{R}^n\setminus f^{-1}(K)$. An affine map which vanishes on a nonempty open set vanishes identically. Thus $h_{\tau}=f$ everywhere. From $h_\tau=\Phi_\tau\circ f$ and the surjectivity of $f$, this implies $\Phi_\tau=\operatorname{id}$, which contradicts with the choice of $\tau$. Therefore $h_{\tau}$ is not affine.

    Finally, since $ h_\tau\in\mathcal W \subset \mathcal V\cap\mathcal U,$ we have $h_\tau\in\mathcal V$, and Theorem~\ref{thm:whit} implies that $h_\tau$ is Anosov with respect to $\hat g$ and has the same stable index as $f$.
    \end{proof}


\noindent \textbf{Z. Li}\\
School of Mathematics (Zhuhai)\\
Sun Yat-Sen University \\
519082, PR China\\
E-mail: lizixu@mail.sysu.edu.cn\\
\ \\
\noindent \textbf{A. Rojas}\\
Hangzhou International Innovation Institute of Beihang University, \\
Hangzhou, 311115, China.\\
Email: tchibatze@gmail.com\\
\ \\
\textbf{S. Roma\~na}\\
School of Mathematics (Zhuhai)\\
Sun Yat-Sen University \\
519082, PR China\\
E-mail: sergio@mail.sysu.edu.cn\\

\end{document}